\documentclass{article}
\usepackage{graphicx} 
\usepackage{amssymb}
\usepackage{amsthm}
\usepackage{mathtools}
\usepackage[pdftex,bookmarks=true]{hyperref}
\usepackage{bbm}

\DeclareMathOperator{\diag}{diag}
\DeclareMathOperator{\tr}{tr}

\newtheorem{thm}{Theorem}
\newtheorem{prop}{Proposition}
\newtheorem{Lem}{Lemma}
\newtheorem{df}{Definition}

\newcommand{\biggset}[1]{\biggl\{\,#1\,\biggr\}}

\usepackage{authblk}
\title{The argument shift method in universal enveloping algebra $U\mathfrak{gl}_d$}
\author[$\dagger$]{Yasushi Ikeda}
\author[$\dagger$,$\ddagger$]{Gerogy Sharygin}
\affil[$\dagger$]{\textit{Dept. of Mechanics and Mathematics,
Lomonosov Moscow State University,
Leninskie gory, d. 1, 119991, Moscow, Russia}}

\vspace{2mm}
\affil[$\ddagger$]{\textit{Kurchatovskiy Institut-ITEP, pl. Akademika Kurchatova, d. 1, 123182, Moscow, Russia}}
\date{May 2023}

\begin{document}

\maketitle

\begin{abstract}
    We prove the conjecture that allows one extend the argument shifting procedure from symmetric algebra $S\mathfrak{gl}_d$ of the Lie algebra $\mathfrak{gl}_d$ to the universal enveloping algebra $U\mathfrak{gl}_d$. Namely, it turns out that the iterated quasi-derivations of the central elements in $U\mathfrak{gl}_d$ commute with each other. Here quasi-derivation is a linear operator on $U\mathfrak{gl}_d$, constructed by Gurevich, Pyatov and Saponov. This allows one better understand the structure of \textit{argument shift algebras} (or \textit{Mishchenko-Fomenko algebras}) in the universal enveloping algebra of $\mathfrak{gl}_d$. 
\end{abstract}
\section{Introduction}
Let $\mathfrak g$ be a Lie algebra, and $\xi\in\mathfrak g$ a (regular) element. Recall that the symmetric algebra $S\mathfrak g$ bears a natural structure of Poisson algebra: we interpret $S\mathfrak g$ as the algebra of polynomial functions on $\mathfrak g^*$ and use the Lie brackets in $\mathfrak g$ to define of Poisson brackets on linear functions. This natural Poisson structure bears the name of \textit{Lie-Poisson structure} or \textit{Kirillov-Kostant-Sourieau structure}. 

\textit{Argument shift subalgebra $A_\xi(\mathfrak g)$ in $S\mathfrak g$} is the Poisson commutative subalgebra (with respect to the Lie-Poisson structure), spanned by the iterated derivations of the central elements in $S\mathfrak g$ with respect to a \textit{constant} vector field $\partial_\xi$: if we choose (and fix) a basis $X_1,\dots, X_N$ in $\mathfrak g$, so that $\xi=\xi^1 X_1+\dots+\xi^N X_N$, then 
\[
\partial_\xi=\sum_{k=1}^N\xi^k\frac{\partial}{\partial x^k},
\]
where $(x^1,\dots,x^N)$ is the coordinate system on $\mathfrak g^*$ associated with the basis. Then the following is true: for any two elements $f,\,g$ in the Poisson center $\mathcal Z_\pi(\mathfrak g)$ of $S\mathfrak g$ and any two natural numbers $p,q$ the Poisson bracket of $\partial_\xi^pf$ and $\partial_\xi^qg$ vanishes:
\[
\{\partial_\xi^pf,\partial_\xi^qg\}=0.
\]

This fact was first observed by Manakov in the study of multi-dimensional rigid body equations (\cite{Man}) for $\mathfrak g=\mathfrak{so}_n$, and was later proved in full generality by Mishchenko and Fomenko (\cite{MiFo}), therefore the algebras $A_\xi$ are in many cases called \textit{Mishchenko-Fomenko algebras}. Later Vinberg \cite{Vin} observed that this construction can be easily generalised to the situation, when the $\mathfrak g^*$ and the vector field $\partial_\xi$ are replaced by an arbitrary Poisson manifold $(M,\,\pi)$ (where $\pi$ is the Poisson bivector) and a vector field $\xi$ on $M$, for which the identity $\mathcal L_\xi^2\pi=0$ holds: as before the equation $\{\partial_\xi^pf,\partial_\xi^qg\}=0$ holds for all Casimir functions $f,g$ and all $p,q\ge0$.

On the other hand, every Poisson structure gives rise to a noncommutative associative $\star$-product on smooth functions on $M$; in particular if $M=\mathfrak g^*$ and we endow $S\mathfrak g$ with the $\star$-product, we obtain the universal enveloping algebra $U\mathfrak g$. Alternatively, the Poincar\'e-Birkhoff-Witt theorem shows that the graded algebra, associated with $U\mathfrak g$ is isomorphic to $S\mathfrak g$. Hence Vinberg asked the question, \textit{whether one can raise the argument sift subalgebra $A_\xi(\mathfrak g)$ from $S\mathfrak g$ to a commutative subalgebra in $U\mathfrak g$?}

At present there are several constructions that give answers to this question, see \cite{Tar},\,\cite{Ryb},\,\cite{Mol},\,\cite{Yakimova}. These papers construct commutative subalgebras in $U\mathfrak g$, that ``raise'' $A_\xi(\mathfrak g)$. All these constructions are based on the consideration of one or another set of generators of the Poisson center in $S\mathfrak g$ (usually these generators are related with the coefficients of the characteristic polynomial), which are then raised to $U\mathfrak g$ by means of one or another constructions. Most of these constructions are based on the considerations of certain ``large'' algebras $Y$ (either Yangians, or the universal enveloping algebras of Kac-Moody algebras at critical level), that are mapped into $U\mathfrak g$. Then one can define the argument shift subalgebra $\mathcal A_\xi(\mathfrak g)$ in $U\mathfrak g$ as the image in $U\mathfrak g$ of suitable commutative subalgebras in $Y$.

On the other hand the actual mechanism of how these algebras emerge in $U\mathfrak g$ still remains quite obscure: apart from the consideration of various particular choices of generators in $\mathcal Z_\pi(\mathfrak g)$ and hence in $A_\xi(\mathfrak g)$ and then finding the way to raise them into the universal enveloping algebra (one needs to keep in mind that certain choices of generators in $A_\xi(\mathfrak g)$ can be more favorable than the others for the purpose of raising into $U\mathfrak g$), no general principles for raising a generic element of the argument shift algebra into the universal enveloping algebra has been given. It is this general approach which we endeavor to give in this paper.

The starting point of our research is pretty na\"\i ve: we ask, if there exists an operation $\hat\partial_\xi:U\mathfrak g\to U\mathfrak g$, that would ``raise'' into the universal enveloping algebra the derivation $\partial_\xi$ from above and such that the iterative application of $\hat\partial_\xi$ to the central elements of $U\mathfrak g$ gives a commutative subalgebra $\mathcal{A}_\xi(\mathfrak g)\subseteq U\mathfrak g$. It turns out, that at least in case when $\mathfrak g=\mathfrak{gl}_d$ (here $\mathfrak{gl}_d$ is either real, or complex general linear Lie algebra: our conclusions do not depend on the choice of ground field) such operation does in fact exist: it is sufficient to take 
\[
\hat\partial_\xi=\sum_{i,j=1}^d\xi^i_j\hat\partial_i^j,
\]
where $\xi=(\xi_j^i)$ is the chosen element in $\mathfrak{gl}_d$, and $\hat\partial_i^j$ are the \textit{quasi-derivations} on $U\mathfrak{gl}_d$, introduced by Gurevich, Pyatov and Saponov (see \cite{GPS} and section 2.3 below). 

Namely, in this paper we prove the following statement: 

\begin{thm}\label{thm:main}
For all $\xi\in\mathfrak{gl}_d$, all $f,g\in\mathcal{Z}(U\mathfrak{gl}_d)$ and all natural $p,q$ the elements $\hat\partial_\xi^pf,\,\hat\partial_\xi^qg$ commute with each other. 
\end{thm}
\noindent This statement can be extended in a natural way to embrace any semisimple Lie algebra $\mathfrak g$ and any element $\xi$ in it, but we postpone the discussion of this and other generalisations of theorem 1 to a forthcoming paper.

The rest of the paper is organised as follows. In the next section  (section 2) we recall basic definitions and results from the theory of the universal enveloping algebras, including the definition of quasi-derivations and the list of their properties needed for our proof, as well as the key property (theorem \ref{thm:Ryb}) of the Mishcheko-Fomenko subalgebras $\mathcal{A}_\xi(\mathfrak g)$ in $U\mathfrak g$, needed for the proof. Section 3 contains the proof of the main theorem 1; it is based on the key result from theortem \ref{thm:Ryb} and is obtained by double induction. Finally, section 4 is given to the discussion of possible generalisations and corollaries of this statement.

\section{Definitions and preliminary results}
\subsection{Universal enveloping algebra $U\mathfrak{gl}_d$}
Let $d\ge1$ be an integer. The Lie algebra $\mathfrak{gl}_d$ over a field $\mathbb K$ (here we assume $\mathbb K=\mathbb R$ or $\mathbb C$) can be defined as the linear space with basis $e_j^i,\,i,j=1,\dots,d$, and the Lie bracket given by the formula
\begin{equation}\label{eq:gldrelations}
[e_j^i,e_q^p]=\delta_j^pe^i_q-e_j^p\delta^i_q.
\end{equation}
Then the universal enveloping algebra $U\mathit{gl}_d$ of $\mathit{gl}_d$ can be defined as the associative $\mathbb K$-algebra equal to the quotient of the free unital associative algebra $\mathbb{K}\langle e^i_j\rangle$ over the field $\mathbb K$ on generators $e^i_j$ by relations \eqref{eq:gldrelations}
\begin{equation}
U\mathit{gl}_d=\mathbb{K}\langle e^i_j\mid i,j=1,\dots,d\rangle/\left\langle e^{i}_{j}e^{p}_{q}-e^{p}_{q}e^{i}_{j}-\delta^{p}_{j}e^{i}_{q}+e^{p}_{j}\delta^{i}_{q}\right\rangle
\end{equation}
The Lie algebra $\mathfrak{gl}_d$ is embedded into $U\mathfrak{gl}_d$ as the linear span of the variables $e_j^i$. 

Below we will use few standard identities on the elements in $U\mathfrak{gl}_d$; in order to formulate them, it is convenient to introduce the matrix $e$ of generators, i.e. $e=(e_j^i)$, where $i$ enumerates the rows and $j$ enumerates the columns, so that
\[
(e^p)^k_\ell=\sum_{i_1,i_2,\dots,i_{p-1}=1}^de^k_{i_1}e^{i_1}_{i_2}\dots e^{i_{p-1}}_\ell.
\]
Then we have:
\begin{align}\label{eq:three}
\bigl[(e^p)^i_j,e^k_\ell\bigr]=\bigl[e^i_j,(e^p)^k_\ell\bigr]=\delta^k_j(e^p)^i_\ell-(e^p)^k_j\delta^i_\ell,&&\forall i,\forall j,\forall, k,\forall\ell
\end{align}
for any nonnegative integer $p$; this equation can be proved by induction on $p$. More generally:
\begin{equation}\label{eq:commrelimp}
\bigl[(e^m)^i_j,(e^n)^k_\ell\bigr]=\sum_{a=1}^{\min(m,n)}\left((e^{a-1})^k_j(e^{m+n-a})^i_\ell-(e^{m+n-a})^k_j(e^{a-1})^i_\ell\right).
\end{equation}
This can again be proved by induction.

The group $GL_d=GL_d(\mathbb K)$ of invertable $d\times d$ matrices acts on $U\mathfrak{gl}_d$ by conjugations and the center of $U\mathfrak{gl}_d$ is equal to the subspace of $GL_d$-invariant elements in $U\mathfrak{gl}_d$. This center is equal to the free commutative algebra with $d$ generators, which can be chosen for instance as $\tau_k=\tr(e^k),\,k=1,\dots,d$, or in more explicit terms:
\[
\tau_k=\sum_{i_1,i_2,\dots,i_k=1}^de^{i_1}_{i_2}e^{i_2}_{i_3}\dots e^{i_k}_{i_1}.
\]
The algebra $U\mathfrak{gl}_d$ is closely related with the symmetric algebra $S\mathfrak{gl}_d$: one can regard $S\mathfrak{gl}_d$ as the algebra of polynomials in variables $e_j^i$; below we will denote by $e$ the matrix of variables $e_j^i$ in $S\mathfrak{gl}_d$, when this does not cause confusion. 

In fact, the natural grading in the free algebra $\mathbb K\langle e^i_j\rangle$ induces a filtration on $U\mathfrak{gl}_d$. Then by Poincar\'e-Birkhoff-Witt (PBW) theorem the graded algebra associated with this filtration is isomorphic to $S\mathfrak{gl}_n$; one can now use the filtration to induce the Poisson structure on $S\mathfrak{gl}_d$: it is the unique Poisson structure, for which
\[
\{e_j^i,e_q^p\}=\delta_j^pe^i_q-\delta_q^ie^p_j.
\]
The same PBW theorem guaranties that the symmetrisation map, i.e. the unique linear map $\sigma:S\mathfrak{gl}_d\to U\mathfrak{gl}_d$, which sends $x^n$ to $x^n$ for any linear combination of variables $x$, induces linear isomorphism of $S\mathfrak{gl}_d\cong U\mathfrak{gl}_d$. In effect, this map not only induces linear isomophism, it also intertwines the action of $GL_d$ on both sides and by virtue of this induces a (linear) isomorphism of centers 
\[
\sigma:\mathcal Z_\pi(S\mathfrak{gl}_d)\cong\mathcal Z(U\mathfrak{gl}_d).
\]
Here $\mathcal Z_\pi(S\mathfrak{gl}_d)$ and $\mathcal Z(U\mathfrak{gl}_d)$  denote the Poisson center of $S\mathfrak{gl}_d$ and the usual center of the universal enveloping algebra. This isomorphism allows one choose other generators of $\mathcal Z(U\mathfrak{gl}_d)$; for instance if we consider the universal characteristic polynomial in $S\mathfrak{gl}_d$:
\[
\det(e-\lambda\mathbbm 1)=\sum_{k=0}^d(-1)^{d-k}\lambda^{d-k} I_{k},
\]
then the elements $I_k\in S\mathfrak{gl}_d$ are generators in the Poisson center $\mathcal Z_\pi(S\mathfrak{gl}_d)$. Their symmetrisations $\tilde I_k=\sigma(I_k)$ are the generators of $\mathcal Z(U\mathfrak{gl}_d)$.

\subsection{Argument shift method}
Let $\xi$ be an element in $\mathfrak{gl}_d$; one can regard it as the numerical matrix $\xi=(\xi_i^j)\in M(d,\mathbb K)$. The following theorem was first proven by Fomenko and Mishchenko, \cite{MiFo}:
\begin{thm}\label{thm:MiFo}
Let $\partial_\xi=\sum_{i,j=1}^d\xi_i^j\partial_j^i$ be the derivation of the polynomial algebra $S\mathfrak{gl}_d$, where
\[
\partial_j^i=\frac{\partial}{\partial e_i^j},\ \mbox{so that}\ \partial_j^i(e^p_q)=\delta_j^p\delta^i_q;
\]
then for any $f,g\in\mathcal Z_\pi(S\mathfrak{gl}_d)$ and any nonnegative integers $p,q$ we have
\[
\{\partial_\xi^pf,\partial_\xi^qg\}=0.
\]
\end{thm}
In effect, Mishchenko and Fomenko proved the same theorem for arbitrary Lie algebra $\mathfrak g$ and a constant vector field $\partial_\xi$ on its dual space $\mathfrak g^*$ (we recall that $S\mathfrak g$ should be regarded as the space of polynomial functions on $\mathfrak g^*$). We will deal with this general case in a forthcoming paper.

It follows from this theorem that the following subspace in $S\mathfrak{gl}_d$
\[
A_\xi(\mathfrak{gl}_d)=\{\partial_\xi^pf\mid f\in\mathcal Z_\pi(S\mathfrak{gl}_d),\,p=0,1,2,\dots\}
\]
is in effect a Poisson commutative subalgebra. It is not difficult to see that if the element $\xi\in\mathfrak{gl}_d$ is \textit{regular} (i.e. if its centralizer has dimension $d$) then the subalgebra $A_\xi(\mathfrak{gl}_d)$ is maximal (Poisson) commutative subalgebra in $S\mathfrak{pl}_d$. This is a free algebra, generated for instance by the elements $\partial_\xi^pI_k,\,p=0,\dots,k-1,\,k=1,\dots,d$. The same statement holds verbatim for any Lie algebra $\mathfrak g$ (provided we choose some set of generators of $\mathcal Z_\pi(S\mathfrak g)$); the algebras $A_\xi(\mathfrak g)$ are called the \textit{argument shift (sub)algebras} or \textit{Mishchenko-Fomenko algebras}.

The algebras $A_\xi(\mathfrak g)$ have been studied a lot in the last thirty to forty years. These constructions were used to describe maximal commutative subalgebras in the algebras of polynomial functions on $\mathfrak g^*$, they were shown to give involutive families of first integrals in the case of many important dynamical systems. A remarkable property of these algebras is that they are determined by their quadratic part. Namely, without loss of generality we can assume that the regular element $\xi$ is given by diagonal matrix $\xi=\mathrm{diag}(\xi_1,\dots,\xi_d)$, so that $\xi_i\neq\xi_j$ if $i\neq j$. Then it has been shown by Vinberg \cite{Vin} and Rybnikov \cite{Ryb2} that \textit{the algebra $A_\xi(\mathfrak{gl}_d)$ coincides with the Poisson centraliser of the subalgebra, generated by diagonal matrices and the elements $T_i(\xi),\,i=1,\dots,d$:
\[
T_i(\xi)=\sum_{j\neq i}\frac{e^j_ie^i_j}{\xi_i-\xi_j}.
\]}
A similar statement holds for arbitrary semisimple Lie algebra $\mathfrak g$.

In the same paper \cite{Vin} Vinberg asked the question, whether the algebras $A_\xi(\mathfrak g)$ can be ``raised'' to a commutative subalgebra $\mathcal A_\xi(\mathfrak g)$ in the universal enveloping algebra $U\mathfrak g$? In the time that passed since 1990, several constructions for this purpose have been suggested: first, in 2000 A. Tarasov (see \cite{Tar}) showed that when $\mathfrak g=\mathfrak{gl}_d$, the symmetrisations of the elements $\partial_\xi^pI_k$ do in fact commute in $U\mathfrak{gl}_d$ (this however, is not true for arbitrary $f,g\in\mathcal Z_\pi(S\mathfrak{gl}_d)$). However there are no evident ways to extend this construction from $\mathfrak{gl}_d$ to other Lie algebras. The first general construction to this end was proposed by Rybnikov in \cite{Ryb}; in this approach the algebra $\mathcal A_\xi(\mathfrak g)$ appears as the image under a suitable homomorphism of the Feigin-Frenkel center in the Kac-Moody algebra at the critical level. This general construction is now used in most papers on this subject; there are explicit formulas for the generators of the corresponding algebras for all classical series of Lie groups, see \cite{Mol},\,\cite{Yakimova}.

Algebras $\mathcal A_\xi(\mathfrak g)$, constructed in this way share many properties of their ``classical'' counterparts: in particular, if $\xi$ is regular, they are maximal commutative and are determined by their quadratic part. Below we will use this property for $\mathfrak{gl}_d$: in \cite{Ryb2} the following proposition is proved:
\begin{thm}\label{thm:Ryb}
When $\xi=\mathrm{diag}(\xi_1,\dots,\xi_d),\,\xi_i\neq\xi_j$, the algebra $\mathcal A_\xi(\mathfrak{gl}_d)$ coincides with the centraliser of the subspace in $U\mathfrak{gl}_d$, spanned by diagonal matrices and the elements
\[
\hat T_i(\xi)=\sum_{j\neq i}\frac{e^j_ie^i_j}{\xi_i-\xi_j}.
\]
\end{thm}
In other words, these are the same elements as before but this time we regard $e^i_j$ as the generators of $U\mathfrak{gl}_d$. One can rephrase this proposition so that \textit{the centraliser of these elements (and diagonal matrices) is commutative subalgebra in $U\mathfrak{gl}_d$}. Below we will make use of this fact. Also remark that a similar statement holds for arbitrary semisimple Lie algebra.

\subsection{Quasiderivatives in $U\mathfrak{gl}_d$}
Our main interset in this paper is to extend to the universal enveloping algebra $U\mathfrak{gl}_d$ the theorem \ref{thm:MiFo}. To this end we need to ``raise'' the partial derivatives from $S\mathfrak{gl}_d$ onto this algebra, which is not easy, since the universal enveloping algebra is noncommutative. One evident way to do it is by pulling these operations to $U\mathfrak{gl}_d$ from symmetric algebra with the help of a linear isomorphism which exists by the PBW theorem (for example, with the help of symmetrisation). However this approach fails, since the inverse of the symmetrisation map $\sigma$ is hard to describe, so it is difficult to describe properties of the operations obtained in this way.

In this paper we are making use of another suitable candidate for the operations, that extend the partial derivatives from $S\mathfrak{gl}_d$ to $U\mathfrak{gl}_d$. These arew the ``quasi-derivations'', introduced earlier by Gurevich, Pyatov and Saponov, see \cite{GPS}. There are different ways to define these operations, they seem to be closely related with many important constructions on Lie algebras. For our purposes it is sufficient to use the following definition (see \cite{mypaper23})
\begin{df}
Partial quasiderivations are the operators $\hat\partial^i_j:U\mathfrak{gl}_d\to U\mathfrak{gl}_d$, uniquely determined by the following conditions:
\begin{enumerate}
    \item $\hat\partial^i_j(1)=0$;
    \item $\hat\partial^i_j(e^p_q)=\delta_j^p\delta^i_q$;
    \item
    \begin{equation}\label{eq:twleibrule}
\hat\partial^i_j(fg)=\hat\partial^i_jf\,g+f\,\hat\partial^i_jg+\sum_{k=1}^d\hat\partial^k_jf\,\hat\partial^i_kg.
    \end{equation}
\end{enumerate}
\end{df}
It is easy to prove that these operations are well-defined and that 
\[
\hat\partial^a_b\hat\partial^c_d=\hat\partial^c_d\hat\partial^a_b.
\]
In many situations it is convenient to write the operations $\hat\partial^i_j$ in the form of the matrix $\hat D=\hat\partial_i^j E^i_j:U\mathfrak{gl}_d\to M_d(U\mathfrak{gl}_d)$ (here $M_d(U\mathfrak{gl}_d)$ denotes the algebra of $d\times d$ matrices with entries in $U\mathfrak{gl}_d$) where $E^i_j$ is the matrix unit, i.e. the matrix all whose entries are equal to $0$ except for the element at the intersection of the $j$-th column and the $i$-th row, which is equal to $1$\footnote{In this notation $e=\sum_{i,j=1}^de^i_jE^i_j$.}. In this notation the equality \eqref{eq:twleibrule} turns into
\begin{equation}\label{eq:leibnizmatrix}
\hat D(fg)=\hat Df\,g+f\,\hat Dg+\hat Df\,\hat Dg.
\end{equation}
This equation is often convenient to write down the computations in a concise form. For instance, we have the following (see \cite{Yasushi22}):
\begin{prop}\label{prop:formula}
Quasiderivations of the matrix elements of the powers of the generating matrix $e$ of $U\mathfrak{gl}_d$ are given by the formula
\begin{equation}\label{eq:importformula}
\hat D\left((e^n)^i_j\right)=\sum_{m=0}^{n-1}\Bigl(f^{(n-m-1)}_+(e)^i(e^m)_j+f^{(n-m-1)}_-(e)(e^m)^i_j\Bigr)^T,
\end{equation}
where ${}^T$ is the transpositiona and for a $d\times d$ matrix $A$ we denote by $A^1,\dots, A^d$ its columns and by $A_1,\dots A_d$ its rows; and $f^{(n)}_\pm(x),\,n=0,1,\dots$ are the polynomials
\begin{equation}\label{eq:polynomials}
f^{(n)}_\pm(x)=\frac{(x+1)^n\pm(x-1)^n}2=\sum_{m=0}^n\frac{1\pm(-1)^{n-m}}2\binom{n}mx^m.
\end{equation}
We call the operator $\hat D$ the \textbf{matrix-valued quasi-derivation}. 
\end{prop}
In what follows we will consider the operator 
\[
\hat\partial_\xi=\sum_{i,j=1}^d\xi_i^j\hat\partial^i_j,
\]
where $\xi_i^j$ are some numerical coefficients; we will denote by $\xi=\xi_i^j E^i_j$ the corresponding matrix; then 
\[
\xi\hat D=(\xi_i^j E^i_j)(\hat\partial_p^q E^p_q)=\xi^j_i\hat\partial_j^q E^i_q, 
\]
so $\hat\partial_\xi=\tr(\xi\hat D)$. We will call this operator \textit{the directional quasi-derivation} on $U\mathfrak{gl}_d$.

One can use the formula \eqref{eq:importformula} from proposition \ref{prop:formula} and the formula \eqref{eq:three} to show that the image of the operator $\hat\partial_\xi$ on central elements lie in the module over the center of $U\mathfrak{gl}_d$ generated by the elements $\tr(\xi^i_j(e^k)_i^j)=\tr(\xi(e^k)^T),\,k=1,2,\dots$ (in fact, they generate this module), i.e. that for any central element $f\in\mathcal Z(U\mathfrak{gl}_d)$ its quasiderivation $\hat\partial_\xi f$ is of the form $\sum_{k=0}^N a_k\tr(\xi(e^k)^T)$, where $a_k\in\mathcal Z(U\mathfrak{gl}_d)$. This can be done by straightforward computations with the help of formula \eqref{eq:importformula}, if we represent the central elements in terms of generators $\tau_k$. We will need this fact below, so let us sketch a proof here (independent of this formula).

First, we show that $\hat\partial^p_q((e^n)^i_j)$ can be written in one of the following forms 
\[
\begin{aligned}
&\sum_{0\le s+t\le n-1}\left(a(n)_{s,t}(e^s)^p_j(e^t)^i_q+b(n)_{s,t}(e^s)^p_q(e^t)^i_j\right)\\
&\sum_{0\le s+t\le n-1}\left(\alpha(n)_{s,t}(e^s)^i_q(e^t)^p_j+\beta(n)_{s,t}(e^s)^i_j(e^t)^p_q\right)
\end{aligned}
\]
for some numerical constants $a(n)_{s,t},\,b(n)_{s,t}$ and $\alpha(n)_{s,t},\,\beta(n)_{s,t}$. This can be done by induction (we assume summation over repeating indices):
\[
\begin{aligned}
\hat\partial^p_q((e^{n+1})^i_j)&=\hat\partial^p_q((e^n)^i_ke^k_j)\\
&=\hat\partial^p_q((e^n)^i_k)e^k_j+(e^n)^i_k\hat\partial^p_q(e^k_j)+\hat\partial^p_\ell((e^n)^i_k)\hat\partial^\ell_q(e^k_j)\\
&=\sum_{s,t}\left(a(n)_{s,t}(e^s)^p_k(e^t)^i_qe^k_j+b(n)_{s,t}(e^s)^p_q(e^t)^i_ke^k_j\right)\\
&\quad+(e^0)^p_j(e^n)^i_q+\hat\partial^p_j((e^n)^i_q).
\end{aligned}
\]
Now by \eqref{eq:three}
\[
\begin{aligned}
(e^s)^p_k(e^t)^i_qe^k_j&=(e^{s+1})^p_j(e^t)^i_q+(e^s)^p_k[(e^t)^i_q,e^k_j]\\
&=(e^{s+1})^p_j(e^t)^i_q+(e^s)^p_k(\delta^k_q(e^t)^i_j-(e^t)^k_q\delta^i_j)\\
&=(e^{s+1})^p_j(e^t)^i_q+(e^s)^p_q(e^t)^i_j-(e^{s+t})^p_q(e^0)^i_j
\end{aligned}
\]
Finally, by inductive hypothesis:
\[
\hat\partial^p_j((e^n)^i_q)=\sum_{0\le s+t\le n-1}\left(a(n)_{s,t}(e^s)^p_q(e^t)^i_j+b(n)_{s,t}(e^s)^p_j(e^t)^i_q\right).
\]
Summing up these three terms we conclude that 
\[
\hat\partial^p_q((e^{n+1})^i_j=\sum_{0\le s+t\le n}\left(a(n+1)_{s,t}(e^s)^p_j(e^t)^i_q+b(n+1)_{s,t}(e^s)^p_q(e^t)^i_j\right)
\]
for some numerical constants $a(n+1)_{s,t},\,a(n+1)_{s,t}$. Similarly we show the existence of $\alpha(n)_{s,t}$ and $\beta(n)_{s,t}$. Now we can show that $\hat D(f)$ can be written in the form
\[
\hat D(f)=\sum_k a_k(f)(e^k)^T
\]
for some central elements $a_k$ (here $e^T$ denotes the transposed matrix): for any $\tau_n=(e^n)^i_i$ we have
\[
\begin{aligned}
\hat D(\tau_n)&=\sum_{p,q=1}^d\hat\partial^p_q(\tau_k)E_p^q\\
&=\sum_{p,q=1}^d\sum_{s,t=0}^{n-1}\left(a(n)_{s,t}(e^{s+t})^p_q+b(n)_{s,t}\tau_t(e^s)^p_q\right)E_p^q\\
&=\sum_{k=0}^{n-1} a_k(\tau_n)(e^k)^p_qE_p^q.
\end{aligned}
\]
The general case now follows by induction from this observation, matrix Leibniz rule \eqref{eq:leibnizmatrix} and the relation \eqref{eq:commrelimp}:
\[
(e^k)^s_q(e^\ell)^p_s=(e^{k+\ell})^p_q+\sum_{a=1}^{\min{k,\ell}}\left(\tau_{k+\ell-a}(e^{a-1})^p_q-\tau_{a-1}(e^{k+\ell-a})^p_q\right).
\]
Finally
\[
\hat\partial_\xi(f)\!=\!\tr(\xi\hat D(f))\!=\!\tr\left(\!\left(\xi^i_jE_i^j\right)\!\left(\sum_k a_k(f)(e^k)^p_qE_p^q\right)\!\right)\!\!=\!\!\sum_k\! a_k(f)\tr(\xi^i_j(e^k)^j_qE^q_i).
\]
It now follows that for any coefficient matrix $\xi$ the operator $\hat\partial_\xi$
sends the any $f,\,g\in\mathcal Z(U\mathfrak{gl}_d)$ into a pair of commuting elements, i.e.
\[
[\hat\partial_\xi f,\hat\partial_\xi g]=0,\ \forall f,g\in\mathcal Z(U\mathfrak{gl}_d).
\]
Indeed by a straightforward computation (see formula \eqref{eq:commrelimp}) one can show that $[\tr(\xi e^m),\tr(\xi e^n)]=0$. However, already the next step, i.e. when we consider the elements $\hat\partial_\xi^2 f$, is much more complicated and explicit formulas for the second quasi-derivative of central elements are pretty hard to obtain, and when obtained, they are quite cumbersome, so commutation relations are hardly visible, see \cite{Yasushi23}. 

Before we proceed to the proof of the general case of theorem \ref{thm:main}, let us make an important remark, concerning the operations $\hat\partial_\xi$. The action of $GL_d$ on $U\mathfrak{gl}_d$ (given by the conjugation of $e$) can be extended to these operations, so that
\[
(\hat\partial_\xi^p f)^g=\hat\partial^p_{\xi^g} f^g
\]
for all $f\in U\mathfrak{gl}_d,\,p=1,2,\dots$, where $\xi^g=g^{-1}\xi g$. In particular, if $f\in\mathcal Z(U\mathfrak{gl}_d)$ we have 
\[
(\hat\partial_\xi^p f)^g=\hat\partial^p_{\xi^g} f.
\]
Since diagonalizable matrices with simple spectrum are dense in the space of (real or complex) matrices, we conclude that the general statement of the theorem \ref{thm:main} follows from the situation, when $\xi=\mathrm{diag}(\xi_1,\dots,\xi_d)$. Thus without loss of generality we can assume that $\xi$ is diagonal and $\xi_i\neq\xi_j$ when $i\neq j$.

\section{Proof of the theorem \ref{thm:main}}
Suppose that $\xi=\diag(\xi_1,\dots,\xi_d)$ is a diagonal numerical matrix with disjoint spectrum. Hence it is regular element in the Lie algebra $\mathfrak{gl}_d$, so we can use the statement of theorem \ref{thm:Ryb}. Now the result would follow if we show that \textit{iterated quasiderivations $\hat\partial_\xi^p f$ of central element $f$ belong to the centralizer of diagonal matrices and the elements $\hat T_i(\xi)$}.

Both conditions (i.e. centrilization of diagonal matrices and of $\hat T_i(\xi)$) now can be proved by induction on $p$. We begin with diagonal matrices: if $p=0$, the condition clearly holds. Now we can assume that $e_i^i\,f=f\,e_i^i$ and apply the operator $\hat\partial_\xi=\sum_{j=1}^d\xi_j\hat\partial_j^j$ to both sides. We get on the left
\[
\begin{aligned}
\hat\partial_\xi(e_i^i\,f)&=\sum_{j=1}^d\xi_j\hat\partial_j^j(e_i^i) f+e_i^i\hat\partial_\xi f+\sum_{j,k=1}^d\xi_j\hat\partial^k_j(e_i^i)\hat\partial_k^j(f)\\
&=\sum_{j=1}^d\xi_j\delta_j^i\delta^j_i f+e_i^i\hat\partial_\xi f+\sum_{j,k=1}^d\xi_j\delta_j^i\delta^k_i\hat\partial_k^j(f)\\
&=e_i^i\hat\partial_\xi f+\xi_i f+\xi_i\hat\partial_i^if.
\end{aligned}
\]
On the right hand side of the original equation, similar computation gives
\[
\hat\partial_\xi(f\,e_i^i)=\hat\partial_\xi f\,e_i^i+\xi_i f+\xi_i\hat\partial_i^if,
\]
and so $e_i^i\,\hat\partial_\xi f=\hat\partial_\xi f\,e_i^i$

Now we turn to the second condition: 
\[
\hat T_i(\xi)\hat\partial_\xi^p f=\hat\partial_\xi^p f\hat T_i(\xi)
\]
for all $i=1,\dots,d$ and $f\in\mathcal Z(U\mathfrak{gl}_d)$.

This is again done by induction: for $p=0$ this is clear. Next, we apply $\hat\partial_\xi$ to both sides of the equation $\hat T_i(\xi)\,f=f\,\hat T_i(\xi)$. We have
\begin{equation}\label{eq:scalar}
    \hat\partial_\xi(\hat T_i(\xi))=\hat\partial_\xi\left(\sum_{j\neq i}\frac{e^j_ie^i_j}{\xi_i-\xi_j}\right)=\sum_{j\neq i}\frac{\xi_j}{\xi_i-\xi_j}
\end{equation}
since
\begin{align}
    \hat\partial^p_q\sum_{j\neq i}\frac{e^j_ie^i_j}{\xi_i-\xi_j}&=\sum_{j\neq i}\frac{\delta^p_i\delta^j_qe^i_j+\delta^p_j\delta^i_qe^j_i+ \delta^j_q\delta^p_j}{\xi_i-\xi_j},\\
\intertext{and hence}
\notag
      \hat\partial_\xi(\hat T_i(\xi))&=\sum_{k=1}^d\sum_{j\neq i}\xi_k\frac{\delta^k_i\delta^j_ke^i_j+\delta^k_j\delta^i_ke^j_i+\delta^j_k\delta^k_j}{\xi_i-\xi_j}\\
      &=\sum_{j\neq i}\frac{\xi_j}{\xi_i-\xi_j}.
\end{align}
For future references let us also give here the following formula
\begin{equation}
\label{eq:second}
    \partial^p_q\partial^r_s\sum_{j\neq i}\frac{e^j_ie^i_j}{\xi_i-\xi_j}=\sum_{j\neq i}\frac{\delta^p_j\delta^i_q\delta^r_i\delta^j_s+\delta^p_i\delta^j_q\delta^r_j\delta^i_s}{\xi_i-\xi_j}.
\end{equation}
To prove the theorem we will need the next Lemma
\begin{Lem}
    \label{prop:center}
    The center of the the universal enveloping algebra is contained in the module
    \begin{equation}\label{eq:module}
        \biggset{x\in U\mathfrak{gl}_d\mid\tr\Bigl(\xi\bigl[\hat D(\hat T_i(\xi)),\hat D x\bigr]\Bigr)=0}.
    \end{equation}
    Here $\xi$ is the diagonal matrix of coefficients.
\end{Lem}
\begin{proof}
    It is sufficient to show that we have
    \begin{equation}
        \tr\Bigl(\xi\bigl[\hat D(T_i(\xi)),(e^n)^T\bigr]\Bigr)=0
    \end{equation}
    for any nonnegative integer $n$ since (see Proposition \ref{prop:formula} and discussion after it) the first matrix-valued quasiderivation $\hat D(x)$ of a central element $x$ is contained in the module over the center generated by the matrices $(e^n)^T$. We have
    \begin{align*}
        \MoveEqLeft\tr\Bigl(\xi\Bigl[\hat D\left(\sum_{j\neq i}\frac{e^j_i
        e^i_j}{\xi_i-\xi_j}\right),(e^n)^T\Bigr]\Bigr)\\
        &=\sum_{k,\ell=1}^d\left(\xi_k\left(\hat\partial^k_\ell\sum_{j\neq i}\frac{e^j_ie^i_j}{\xi_i-\xi_j}\right)(e^n)^\ell_k-\xi_\ell(e^n)^\ell_k\left(\hat\partial^k_\ell\sum_{j\neq i}\frac{e^j_ie^i_j}{\xi_i-\xi_j}\right)\right)\\
        &=\sum_{k,\ell=1}^d\xi_k\left(\sum_{j\neq i}\frac{\delta^k_i\delta^j_\ell e^i_j+\delta^k_j\delta^i_\ell e^j_i+ \delta^j_\ell\delta^k_j}{\xi_i-\xi_j}\right)(e^n)^\ell_k\\
        &\qquad\qquad\qquad-\xi_\ell(e^n)^\ell_k\left(\sum_{j\neq i}\frac{\delta^k_i\delta^j_\ell e^i_j+\delta^k_j\delta^i_\ell e^j_i+ \delta^j_\ell\delta^k_j}{\xi_i-\xi_j}\right)\\
        &=\sum_{j\neq i}\frac{\xi_ie^i_j(e^n)^j_i+\xi_je^j_i(e^n)^i_j+\xi_j(e^n)^j_j-\xi_j(e^n)^j_ie^i_j-\xi_i(e^n)^i_je^j_i-\xi_j(e^n)^j_j}{\xi_i-\xi_j}\\
        &=\sum_{j\neq i}\Bigl(e^i_j(e^n)^j_i-\frac{\xi_j}{\xi_i-\xi_j}\bigl[(e^n)^j_i,e^i_j\bigr]-(e^n)^i_je^j_i+\frac{\xi_j}{\xi_i-\xi_j}\bigl[e^j_i,(e^n)^i_j\bigr]\Bigr)\\
        &=(e^{n+1})^i_i-e^i_i(e^n)^i_i-(e^{n+1})^i_i+(e^n)^i_ie^i_i=0,
    \end{align*}
    where in the last two lines we used the commutation relation \eqref{eq:three}.
\end{proof}

\begin{prop}\label{prop:invariant}
    The module \eqref{eq:module} is invariant under the action of the quasi-derivation $\hat\partial_\xi$.
\end{prop}
\begin{proof}
    Suppose that $x$ is an element of the module \eqref{eq:module}. We have
    \begin{align*}
        0&=\hat\partial_\xi\tr\Bigl(\xi\bigl[\hat D\sum_{j\neq i}\frac{e^j_ie^i_j}{\xi_i-\xi_j},\hat D x\bigr]\Bigr)\\&=
        \begin{multlined}[t]
            \tr\Bigl(\xi\bigl[\hat D\sum_{j\neq i}\frac{e^j_ie^i_j}{\xi_i-\xi_j},\hat D\hat\partial_\xi x\bigr]\Bigr)+\tr\Bigl(\xi\bigl[\hat D\hat\partial_\xi\sum_{j\neq i}\frac{e^j_ie^i_j}{\xi_i-\xi_j},\hat D x\bigr]\Bigr)\\
+\sum_{p,q,k,\ell=1}^d\Bigl(\xi_q\xi_k\left(\hat\partial^p_q\hat\partial^k_\ell\sum_{j\neq i}\frac{e^j_ie^i_j}{\xi_i-\xi_j}\right)\bigl(\hat\partial^q_p\hat\partial^\ell_k x\bigr)\Bigr)\\
            -\xi_p\xi_\ell\bigl(\hat\partial^q_p\hat\partial^\ell_k x\bigr)\left(\hat\partial^p_q\hat\partial^k_\ell\sum_{j\neq i}\frac{e^j_ie^i_j}{\xi_i-\xi_j}\right)
        \end{multlined}
    \end{align*}
    since the quasiderivations commute. We have
    \begin{equation}
        \tr\Bigl(\xi\bigl[\hat D\hat\partial_\xi\sum_{j\neq i}\frac{e^i_je^j_i}{\xi_i-\xi_j},\hat D x\bigr]\Bigr)=0
    \end{equation}
    by the equation \eqref{eq:scalar}. Finally we have
    \begin{multline}
\sum_{p,q,k,\ell=1}^d\Bigl(\xi_q\xi_k\!\left(\hat\partial^p_q\hat\partial^k_\ell\sum_{j\neq i}\frac{e^j_ie^i_j}{\xi_i-\xi_j}\right)\!\bigl(\hat\partial^q_p\hat\partial^\ell_k x\bigr)-\xi_p\xi_\ell\bigl(\hat\partial^q_p\hat\partial^\ell_k x\bigr)\!\left(\hat\partial^p_q\hat\partial^k_\ell\sum_{j\neq i}\frac{e^j_ie^i_j}{\xi_i-\xi_j}\right)\Bigr)\\
        =\sum_{p,q,k,\ell=1}^d\sum_{j\neq i}\bigl(\delta^p_j\delta^i_q\delta^k_i\delta^j_\ell+\delta^p_i\delta^j_q\delta^k_j\delta^i_\ell\bigr)\frac{\xi_q\xi_k-\xi_p\xi_\ell}{\xi_i-\xi_j}\hat\partial^q_p\hat\partial^\ell_k x=0
    \end{multline}
    by the equation \eqref{eq:second}.
\end{proof}

Now it follows by induction that for any central element $x$ of the universal enveloping algebra its iterated quasiderivation $\hat\partial_\xi^m x$ is in the module \eqref{eq:module}, i.e. that
\begin{equation}\label{eq:trace}
    \tr\Bigl(\xi\bigl[\hat D(\hat T_i
(\xi)),\hat D\hat\partial_\xi^mx\bigr]\Bigr)=0
\end{equation}
for all $i=1,\dots,d$ and any nonnegative integer $m$ by Proposition \ref{prop:center} and Proposition \ref{prop:invariant}.

We can finally prove that the second condition also holds:
\begin{prop}
    For any nonnetagive integer $m$ we have
    \begin{equation}
        \left[\hat T_i(\xi),\partial_\xi^mx\right]=0
    \end{equation}
    for all $i=1,\dots,d$ and any central element $x$.
\end{prop}

\begin{proof}
    The proof is by induction on $m$. The base of induction is evident. Suppose now that for some $m>0$ we have
    \begin{equation}
        \Bigl[\hat T_i(\xi),\hat\partial_\xi^{m-1}x\Bigr]=0,
    \end{equation}
    i.e. let the inductive hypothesis hold. Then we have
    \begin{align*}
        0&=\hat\partial_\xi\left[\hat T_i(\xi),\partial_\xi^{m-1}x\right]\\&=
        \begin{multlined}[t]
            \Bigl[\hat\partial_\xi\hat T_i(\xi),\hat\partial_\xi^{m-1}x\Bigr]+\Bigl[\hat T_i(\xi),\hat\partial_\xi^mx\Bigr]\\
            +\tr\Bigl(\xi\Bigl[\hat D(\hat T_i(\xi)),\hat D\hat\partial_\xi^{m-1}x\Bigr]\Bigr)
        \end{multlined}
    \end{align*}
    by the twisted Leibniz rule. But
    \begin{equation*}
        \Bigl[\hat\partial_\xi(\hat T_i(\xi)),\hat\partial_\xi^{m-1}x\Bigr]=0
    \end{equation*}
    by the equation \eqref{eq:scalar} and
    \begin{equation*}
        \tr\Bigl(\xi\bigl[\hat D(\hat T_i(\xi)),\hat D\hat\partial_\xi^{m-1}x\bigr]\Bigr)=0
    \end{equation*}
    by the equation \eqref{eq:trace}.
\end{proof}
Hence if $\xi$ is a diagonal matrix with simple spectrum, then for all $f\in\mathcal Z(S\mathfrak{gl}_d)$ and all $p=0,1,2,\dots$ the elements $\hat\partial_\xi^pf$ lie in the centraliser both of the diagonal subalgebra in $\mathfrak{gl}_d$ and the subspace, spanned by the elements $\hat T_i(\xi)$, and hence by theorem \ref{thm:Ryb} in $\mathcal A_\xi(\mathfrak{gl}_d)$. So, these elements commute with each other. Since regular diagonalizable matrices are dense in $\mathfrak{gl}_d$, the same is true for all $\xi$ and the theorem \ref{thm:main} is proved.

\section{Conclusive remarks and observations}
Let us conclude this paper with a short list of questions, remarks and observations related with the constructions that we used here.

First of all we remark that the theorem we proved allows one to obtain the elements in the algebra $\mathcal A_\xi(\mathfrak{gl}_d)$ in a rather straightforward way, by direct computations: for any element $f\in A_\xi(\mathfrak{gl}_d)\subseteq S\mathfrak{gl}_d$, equal to the iterated derivation $\partial_\xi^pF$ of some element in the center $F\in\mathcal A_\xi(\mathfrak{gl}_d)$, we have an element $\hat f=\hat\partial_\xi^p\sigma(F)\in\mathcal A_\xi(\mathfrak{gl}_d)$ that raises $f$. This method to raise the elements from $A_\xi(\mathfrak{gl}_d)$ to the Mishchenko-Fomenko algebras in $U\mathfrak{gl}_d$, is rather easy compared to previous approaches, which asked for a multy-stage procedure: first one has to express the element $F$ in terms of some standard generators of the Poisson center (usually, in terms of $\tilde I_k$), use this to express $f$ in terms of iterated derivations of the generators, use the standard formulas to raise the iterated derivations into $U\mathfrak{gl}_d$ and restore $\hat f$ from this information.

Next, it is clear that the same statement for the ``quasi-derivations of the second type'', also introduced by Gurevich, Pyatov and Saponov. Namely, in addition to the operations $\hat\partial^i_j$, they introduced the operators $\bar\partial_j^i$, which verify the same set of relations as $\hat\partial_j^i$ (see definition 1), except for the different type of Leibniz rule, \eqref{eq:twleibrule}: we have
 \begin{equation}\label{eq:twleibrule'}
\bar\partial^i_j(fg)=\bar\partial^i_jf\,g+f\,\bar\partial^i_jg-\sum_{k=1}^d\bar\partial^i_kf\,\bar\partial^k_jg.
    \end{equation}
The proof is obtained from the one we presented here in a word-for-word manner, only changing some notation. The elements obtained by using $\bar\partial_j^i$ instead of $\hat\partial_j^i$ will also lie inside the Mishchenko-Fomenko algebra $\mathcal A_\xi(\mathfrak{gl}_d)$, so in particular $[\hat\partial_\xi^pf,\bar\partial_\xi^qg]=0$ for all matrices of coefficients $\xi$, all $f,g\in\mathcal Z(U\mathfrak{gl}_d)$ and all $p,q=0,1,2,\dots$. 

In general, one can associate a family of quasiderivations to any representation of the Lie algebra $\mathfrak{gl}_d$ (see paper \cite{mypaper23}). It looks plausible that these maps will verify a similar  property (either all of them or under some simple conditions), however we postpone consideration of the general situation to a later paper. It is also possible to extend the notion of quasiderivations to the universal enveloping algebra $U\mathfrak g$ of an arbitrary Lie algebra $\mathfrak g$ and arbitrary representation $\rho$ of $\mathfrak g$. The question, when the corresponding elements will commute with each other is postponed to a forthcoming paper.

Finally, let us consider a couple of simple corollaries that are derived from the principal result of the present paper. First of all, recall that in paper \cite{mypaper23} we reduced the question whether the elements $\hat\partial_\xi^pf$ and $\hat\partial_\xi^qg$ commute to the question of commutativity of the following elements in $\mathrm{End}((\mathbb R^n)^{\otimes k})$: for any $P,Q$ in the centers of the group algebras $S_{k+p}$ and $S_{k+q}$ respectively we take the elements
\[
\begin{aligned}
\varphi_\xi&=\tr_{1,\dots,p}((\xi^{\otimes p}\otimes\mathbbm 1^{\otimes k})P)\\
\psi_\xi&=\tr_{1,\dots,q}((\xi^{\otimes q}\otimes\mathbbm 1^{\otimes k})Q)
\end{aligned}
\]
where we let the groups $S_N$ act on the tensor products $(\mathbb R^n)^{\otimes N}$ by permutations of the tensor legs and $\tr_{1,\dots,p}$ denotes the trace with respect to the first $p$ tensor legs. Then $[\varphi_\xi,\psi_\xi]=0$.

Another interesting collection of corollaries that can be derived from the main result of this paper appears when we represent the algebra $U\mathfrak{gl}_d$ by differential operators. For instance, one can embed $U\mathfrak{gl}_d$ into the algebra of differential operators on $GL_d$ as the subalgebra of left- (or right-) invariant operators on the group. For instance, in standard coordinates $x_i^j$ on the group, the generators $e^p_q$ of $\mathfrak{gl}_d$ can be represented by the elements
\[
X^p_q=\sum_{k=1}^dx^p_k\partial^k_q.
\]
Here $\partial_q^p=\frac{\partial}{\partial x^q_p}$ as usual. One now can reinterpret the operators $\hat\partial^j_i$ in terms of commutators of the differential operators on $GL_d$, see \cite{mypaper23}. This gives one a vast series of examples of commuting differential operators in variables $x_i^j$; their relation with other known examples of this sort (see for instance \cite{Orlov23}) is an interesting open question.

\end{document}